\documentclass[a4paper,12pt]{article}

\usepackage{amssymb}
\usepackage{epsfig}
\usepackage{amsfonts}
\usepackage{amsmath}
\usepackage{euscript}
\usepackage{amscd}
\usepackage{amsthm}
\DeclareMathAlphabet{\mathpzc}{OT1}{pzc}{m}{it}

\newtheorem{thm}{Theorem}[section]
\newtheorem{lem}[thm]{Lemma}
\newtheorem{prop}[thm]{Proposition} 
\newtheorem{cor}[thm]{Corollary}
\newtheorem{rem}[thm]{Remark}
\newtheorem{ex}[thm]{Example}

\newcommand{\m}{\mathpzc{m}}
\newcommand{\p}{\mathpzc{p}}

\newcommand{\bQ}{\mathbb Q}
\newcommand{\bR}{\mathbb R}

\newcommand{\td}{\operatorname{tr.deg}}

 \title{Nice derivations over principal ideal domains}
 \author{Nikhilesh Dasgupta and Neena Gupta\\
 {\small{\it Statistics and Mathematics  Unit, Indian Statistical Institute,}}\\
 {\small{\it 203 B.T. Road, Kolkata 700 108, India}}\\
 {\small{\it e-mail:  its.nikhilesh@gmail.com}}\\
 {\small{\it e-mail : neenag@isical.ac.in, rnanina@gmail.com}}}

\begin{document}
\date{}
\maketitle

\abstract{In this paper we investigate to what extent the results of Z. Wang and D. Daigle on ``nice derivations'' 
of the polynomial ring $k[X,Y,Z]$ over a field 
$k$ of characteristic zero extend to the polynomial ring $R[X,Y,Z]$ over a PID $R$, containing the field of rational numbers. 
One of our results shows that the kernel of a nice derivation on $k[X_1,X_2,X_3,X_4]$ of rank at most three 
is a polynomial ring over $k$.

\smallskip

 \noindent
 {\small {{\bf Keywords}. Polynomial Rings, Locally Nilpotent Derivation, Nice Derivation.}
 
 }

\section{Introduction}
By a ring, we will mean a commutative ring with unity. 
Let $R$ be a ring and $n(\geqslant 1)$ be an integer. For an $R$-algebra $A$, we use the notation $A=R^{[n]}$ to 
denote that $A$ is isomorphic to a polynomial ring in $n$ variables over $R$. We denote the group of units of $R$ by $R^{*}$.

Let $k$ be a field of characteristic zero, $R$ a $k$-domain, $B:=R^{[n]}$ and $m$ is a positive integer $\leq n$. In this
paper, we consider locally nilpotent derivations $D$ on $B$, which satisfy $D^{2}(T_i)=0$ for all 
$i \in \{1,\dots,m\} \subseteq \{1,\dots,n\}$ for some coordinate system $(T_1,T_2, \dots ,T_n)$ of $B$. For convenience, we 
shall call such a derivation $D$ as a {\it quasi-nice derivation}. In the case $m=n$, such a $D$ is called a 
{\it nice derivation} (Thus a nice derivation is also a quasi-nice derivation). We investigate the 
rank of $D$ when $n=3$ and $R$ is a PID (see Section $2$ for the definition of rank of $D$). 

The case when $B=k^{[3]}$ was investigated by Z. Wang in \cite{W}. He showed that $rank$ $D$ is less than 3 for
the cases $(m,n)=(2,3),(3,3)$ and that $rank$ $D$ $=1$ when  $D$ is a nice derivation
(i.e., for $(m,n)=(3,3)$). 
In \cite{Da}, Daigle proved that the rank of $D$ is less than $3$ even in 
the case $(m,n)=(1,3)$ (\cite[Theorems 5.1 and 5.2]{Da}). 

Now let $R$ be a Noetherian domain containing $\bQ$, say $R$ is regular. 
It is natural to ask  how far we can extend the above results to $R[X,Y,Z](=R^{[3]})$. 
In particular, we consider the following question for nice derivations.

\smallskip

\noindent
{\bf Question 1.} {\it If $D$ is a nice derivation of $R[X,Y,Z]$, then is rank $D$ $=1$, or, at least, is rank $D$ $<3$?}

\smallskip
 
In Section $3$, we show that when $R$ is a PID, the rank of $D$ is indeed less than 3 (Theorem \ref{PID})
and construct a nice derivation $D$ over $k^{[1]}$ with $rank$ $D=2$ (Example \ref{PIDex}). 
Moreover, we construct a nice derivation $D$ over $k^{[2]}$ with $rank$ $D=3$ (Example \ref{UFDex}) showing that
Theorem \ref{PID} does not extend to two-dimensional regular or factorial domains.

An important open problem in Affine Algebraic Geometry asks whether the kernel of any $D\in LND_{k}(k^{[4]})$ is 
necessarily finitely generated. In the case when $rank~D \leq3$, Bhatwadekar and Daigle had shown that the kernel 
is indeed finitely generated \cite[Theorem 1]{BD}. However Daigle and Freudenburg had constructed an example to 
show that the kernel need not be $k^{[3]}$ \cite{DF01}. Under the additional hypothesis that the kernel 
is regular, Bhatwadekar, Gupta and Lokhande showed that the kernel is indeed $k^{[3]}$ \cite[Theorem 3.5]{BGL}. 
A consequence of Theorem \ref{PID} of our paper is that in the case 
$rank~D \leq3$, the kernel of any nice derivation $D$ is necessarily $k^{[3]}$ (Corollary \ref{4variables}).

The following question on quasi-nice derivations arises in view of Wang's result that $rank$ $D$ is less than $3$ 
for $(m,n)=(2,3)$. 

\smallskip

\noindent
{\bf Question 2.} {\it  If $D$ is a locally nilpotent derivation of $R[X,Y,Z]$, such that $D$ is irreducible 
and $D^{2}X=D^{2}Y=0$, is then $rank$ $D$ $<3?$
}

\smallskip

In Section $4$, we investigate this question and obtain some partial results when $R$ is a PID (Proposition \ref{partI}) and 
a Dedekind domain (Proposition \ref{dd2}). Example \ref{2var} shows that Question 2 has a negative answer in general, even 
when $R$ is a PID. We shall also construct a strictly $1$-quasi derivation (defined in Section $4$) on 
$R^{[3]}$ over a PID $R$ (Example \ref{1var}). By a result of Daigle (quoted in Section $4$ as Theorem \ref{d10}), 
there does not exist such a derivation on $k^{[3]}$, where $k$ is a field of characteristic zero.

\section{Preliminaries}

For a ring $A$ and a nonzerodivisor 
$f \in A$, we use the notation $A_f$ to denote the localisation of $A$ with respect to the multiplicatively closed 
set $\{1,f,f^2,\dots\}$.

Let $A \subseteq B$ be integral domains. Then the 
transcendence degree of the field of fractions of $B$ over that of $A$ is denoted by $\td_{A}{B}$.

A subring $A \subseteq B$ is defined to be {\it factorially closed} in $B$ if, given nonzero 
$f,g \in B$, the condition $fg \in A$ implies $f \in A$ and $g \in A$. When the ambient ring $B$ is understood, 
we will simply say that $A$ is factorially closed. A routine verification shows that a factorially closed subring 
of a UFD is a UFD. If $A$ is a factorially closed subring of $B$, then $A$ is algebraically closed in $B$; 
further if $S$ is a multiplicatively closed set in $A$ then $S^{-1}A$ is a factorially closed subring of $S^{-1}B$.

Let $k$ be a field of characteristic zero, $R$ a $k$-domain, and $B$ an $R$-domain. The set of locally nilpotent 
$R$-derivations of $B$ is denoted by $LND_{R}(B)$. When $R$ is understood from the context (e.g. when $R=k$), 
we simply denote it by $LND(B)$. We denote the kernel of a locally nilpotent derivation $D$ by $Ker$ $D$.

Let $D\in LND_{R}(B)$ and $A:=Ker$ $D$. It is well-known that $A$ is a factorially closed 
subring of $B$ \cite[1.1(1)]{DF}. For any multiplicatively closed subset $S$ of 
$A \setminus \{0\}$, $D$ extends to a locally nilpotent derivation of $S^{-1}B$ with kernel $S^{-1}A$ 
and $B \cap S^{-1}A=A$ \cite[1.1(2)]{DF}. Moreover if $D$ is non-zero, then $\td_{A}B=1$ \cite[1.1(4)]{DF}.
A locally nilpotent derivation $D$ is said to be {\it reducible} 
if there exists a non-unit $b \in B$ such that $DB \subseteq (b)B$; otherwise 
$D$ is said to be {\it irreducible}. An element $s \in B$ is called a {\it slice} if $Ds=1$, and a 
{\it local slice} if $Ds \in A$ and $Ds \neq 0$. Moreover $D$ is said to be {\it fixed-point free} if the $B$ ideal $(DB)=B$. 

When $B:=R^{[n]}$ and $D \in LND_{R}(B)$, the {\it rank} of $D$, denoted by $rank$ $D$, is defined 
to be the least integer $i$ for which
there exists a coordinate system $(X_1,X_2,\dots,X_n)$ of $B$ satisfying $R[X_{i+1},\dots,X_n]\subseteq A$.

Now let $B$ be a $k$-domain and $D$ an element of $LND(B)$ with a local slice $r \in B$. 
The Dixmier map induced by $r$ is defined to be  
the $k$-algebra homomorphism 
$\pi_r:B \rightarrow B_{Dr}$, given by 
\begin{equation*}
\pi_r(f)=\sum_{i\geqslant 0}\frac{(-1)^i}{i!}D^{i}f\frac{r^i}{(Dr)^i}.
\end{equation*}

The following important result is known as the Slice Theorem \cite[Corollary 1.22]{ST}.
\begin{thm}\label{st}
 Let $k$ be a field of characteristic zero and $B$ a $k$-domain. 
 Suppose $D \in LND(B)$ admits a slice $s\in B$, and let $A=Ker$ $D$. Then 
 \begin{enumerate}
  \item[\rm (a)]$B=A[s]$ and $D=\frac{\partial}{\partial s}$.
  \item[\rm (b)]$A=\pi_{s}(B)$ and $Ker$ $\pi_{s}=sB$.
  \item[\rm (c)]If $B$ is affine, then $A$ is affine.
 \end{enumerate}
\end{thm}

The following theorem of Daigle and Freudenburg characterizes locally nilpotent derivations of $R^{[2]}$, 
where $R$ is a UFD containing $\bQ$ \cite[Theorem 2.4]{DF}.
\begin{thm}\label{df}
Let $R$ be a UFD containing $\bQ$ with field of fractions $K$ and let $B=R[X,Y]=R^{[2]}$. 
For an $R$-derivation $D\neq0$ of $B$, the following are equivalent:
\begin{enumerate}
 \item[\rm (i)] $D$ is locally nilpotent.
 \item[\rm (ii)] $D=\alpha(\frac{\partial F}{\partial Y}\frac{\partial}{\partial X}-
 \frac{\partial F}{\partial X}\frac{\partial}{\partial Y})$, for some $F \in B$ which is a variable of $K[X,Y]$
  satisfying\\
  $gcd_B(\frac{\partial F}{\partial X},\frac{\partial F}{\partial Y})=1$, and for some 
 $\alpha \in R[F] \setminus \{0\}$.
 \end{enumerate}
Moreover, if the above conditions are satisfied, then $Ker$ $D=R[F]=R^{[1]}$.
\end{thm}

With the same notation as above, the following lemma gives interesting results when $D$ satisfies some additional 
hypothesis \cite[Lemma 4.2]{W}. 
\begin{lem}\label{UFDWang}
Let $R$ be a UFD containing $\bQ$, $B=R[X,Y](=R^{[2]})$ and $D \in LND_{R}(B)$ such that $D$ is irreducible. Then the 
following hold:
\begin{enumerate}
 \item [\rm (i)] If $D^{2}X=0$, then $Ker$ $D=R[bY+f(X)]$, where $b \in R$ and $f(X) \in R[X]$. Moreover, $DX \in R$ and 
 $DY \in R[X]$.
 \item [\rm (ii)] If $D^{2}X=D^{2}Y=0$, then $D=b\frac{\partial }{\partial X}-a\frac{\partial }{\partial Y}$ for some 
 $a,b \in R$. Moreover, \\ $Ker$ $D=R[aX+bY]$.
 \item [\rm (iii)] If $R$ is a PID and $D^{2}X=D^{2}Y=0$, then $D$ has a slice.
\end{enumerate}

\end{lem}

Over a Noetherian domain containing $\bQ$, a necessary and sufficient condition for the kernel of a nonzero irreducible 
$D \in LND_{R}(R[X,Y])$ to be a polynomial ring 
is given by \cite[Theorem 4.7]{BD97}.
\begin{thm}\label{bd97}
Let $R$ be a Noetherian domain containing $\bQ$ and let $D$ be a non-zero irreducible locally nilpotent derivation of the 
polynomial ring $R[X,Y]$. Then the kernel $A$ of $D$ is a polynomial ring in one variable over $R$ if and only if $DX$ and $DY$ 
either form a regular $R[X,Y]$-sequence or are comaximal in $R[X,Y]$. Moreover if $DX$ and $DY$ are comaximal in $R[X,Y]$, 
then $R[X,Y]$ is a polynomial ring in one variable over $A$.
\end{thm}

An important result on fixed-point free locally nilpotent derivations is the following \cite[Theorem 4.16]{ST}.
\begin{thm}\label{fpf}
Let $R$ be any $\bQ$-algebra, and let $B=R[X,Y]=R^{[2]}$. Given $D \in LND_{R}(R[X,Y])$, the following 
conditions are equivalent:
\begin{enumerate}
 \item [\rm (1)] $D$ is fixed-point free, i.e., $(DB)=B$, where $(DB)$ is the $B$-ideal generated by $DB$.
 \item [\rm (2)] There exists $s \in B$ with $Ds=1$.
\end{enumerate}
In addition, when these conditions hold, $Ker$ $D=R^{[1]}$.
\end{thm}

For a ring containing $\bQ$, the following cancellation theorem was proved by Hamann \cite[Theorem 2.8]{HAM}.
\begin{thm}\label{ham}
Let $R$ be a ring containing $\bQ$ and $A$ be an $R$-algebra such that $A^{[1]}=R^{[2]}$. Then $A=R^{[1]}$.
\end{thm}

The following is a well-known result of Abhyankar, Eakin and Heinzer \cite[Proposition 4.8]{AEH}.
\begin{thm}\label{aeh}
Let $C$ be a UFD and let $X_1,\dots,X_n$ be indeterminates over $C$. Suppose that $A$ is an integral domain of 
transcendence degree one over $C$ and that $C \subseteq A \subseteq C[X_1,\dots,X_n]$. If $A$ is a 
factorially closed subring of $C[X_1,\dots,X_n]$, then $A=C^{[1]}$.
\end{thm}

The following local-global theorem was proved by Bass, Connell and Wright \cite{BCW} and 
independently by Suslin \cite{SUS}.
\begin{thm}\label{bcw}
Let $R$ be a ring and $A$ a finitely presented $R$-algebra. Suppose that for all maximal 
ideals $\m$ of $R$, the $R_{\m}$-algebra $A_{\m}$ is isomorphic to the symmetric algebra of some $R_{\m}$-module. 
Then $A \cong Sym_R{(L)}$ for some finitely presented $R$-module $L$. 
\end{thm}

The following result is known as Serre's Splitting Theorem \cite[Theorem 7.1.8]{SST}. 
\begin{thm}\label{sst}
Let $A$ be a Noetherian ring of finite Krull dimension. Let $P$ be a finitely generated 
projective $A$-module of rank greater than dimension of $A$. Then $P$ has a unimodular element. 
\end{thm}

Following is the famous Cancellation Theorem of Hyman Bass \cite[Theorem 7.1.11]{SST}.
\begin{thm}\label{bass}
Let $R$ be a Noetherian ring of dimension $d$ and $P$ a finitely generated projective $R$-module of rank $>d$. 
Then $P$ is ``cancellative'', i.e., $P \oplus Q \cong P^{'} \oplus Q$ for some finitely generated projective 
$R$-module $Q$ implies that $P \cong P^{'}$.
\end{thm}

We now state a local-global result for a graded ring \cite[Theorem 4.3.11]{SST}.
\begin{thm}\label{gr}
Let $S=S_{0} \oplus S_{1} \oplus S_{2} \dots $ be a graded ring and let $M$ be a finitely 
presented $S$-module. Assume that for every maximal ideal $\m$ of $S_{0}$, $M_{\m}$ is extended from 
$(S_{0})_{\m}$. Then $M$ is extended from $S_{0}$.
\end{thm}

For convenience, we state below an elementary result.
\begin{lem}\label{domain}
Let $A$ and $B$ be integral domains with $A \subseteq B$. If there exists $f$ in $A$, such that
 $A_f=B_f$ and $fB \cap A=fA$, then $A=B$.
\end{lem}

\begin{proof}
Let $b \in B$. Suppose, if possible $b \notin A$. Now since $B_f=A_f$, we have $b \in A_f$. Hence there exist
 $a \in A$ and an integer $n > 0$ such that $b=a/f^{n}$. We may assume that $n$ is the least possible.
 But then $a \in fB \cap A=fA$. Let $a=fa_1$ for some $a_1 \in A$. Then $b=a_{1}/f^{n-1}$, contradicting 
 the minimality of $n$.
\end{proof}

\section{Nice Derivations}
In this section, we shall explore generalisations of the following theorem of Z. Wang \cite[Proposition 4.6]{W}.
\begin{thm}\label{Wang1}
Let $K$ be a field of characteristic zero and 
$K[X,Y,Z]=K^{[3]}$. Suppose that $D(\neq0) \in LND_{K}(K[X,Y,Z])$ satisfies $D^{2}X=D^{2}Y=D^{2}Z=0$. 
Then the following hold:
\begin{enumerate}
\item[\rm (i)]$Ker$ $D$ contains a nonzero linear form of \{$X,Y,Z$\}.
\item[\rm (ii)]$rank$ $D=1$.
\item[\rm (iii)]If $D$ is irreducible, then for some coordinate system ($X^{'},Y^{'},Z^{'}$) of $K[X,Y,Z]$ related to ($X,Y,Z$) 
by a linear change, 
\begin{equation*}
D=f(X^{'})\frac{\partial}{\partial Y^{'}}+g(X^{'})\frac{\partial}{\partial Z^{'}}
\end{equation*}
where $f$, $g \in K[X^{'}]$ and ${\rm gcd}_{K[X^{'}]}(f,g)=1$.
\end{enumerate}
\end{thm}

We first observe the following result.

\begin{lem}\label{UFD}
Let $R$ be a UFD containing $\bQ$ and $D$($\neq0$) $\in LND_{R}(R[X,Y,Z])$, where $R[X,Y,Z]=R^{[3]}$ and $rank$ $D$ $< 3$.
Then $Ker$ $D = R^{[2]}$.
\end{lem}

\begin{proof}
Let $A:=Ker$ $D$.
Since $rank$ $D$ $< 3$, there exists $X^{'}\in R[X,Y,Z]$ such that $R[X,Y,Z]=R[X^{'}]^{[2]}$ and $X^{'} \in A$. Then 
taking $C=R[X^{'}]$, it follows from Theorem \ref{aeh} that $A(=Ker$ $D)=R[X^{'}]^{[1]} = R^{[2]}$.
\end{proof}

The following example shows that Lemma \ref{UFD} does not extend to a Noetherian normal domain $R$ 
which is not a UFD.

\begin{ex}\label{dd}
{\em
Let $\bR[a,b]=\bR^{[2]}$ and $R:=\frac{\bR[a,b]}{(a^{2}+b^{2}-1)}$. Let $B:=R[X,Y,Z]=R^{[3]}$ and 
$D$ be an $R$-linear $LND$ of $B$, such that $$DX=a, \quad DY=b-1 \quad \text{and} \quad DZ=aY+(1-b)X.$$ 
Setting $u=aY+(1-b)X$, $v=(1+b)Y+aX$ and $w=2Z+uY-vX$, we see that $Du=Dv=Dw=0$ and $D^{2}X=D^{2}Y=D^{2}Z=0$. 

Let $A:=Ker$ $D$. 
Now $B_{(1+b)}=R_{(1+b)}[v,w,X]$ and $B_{(1-b)}=R_{(1-b)}[u,w,Y]$. Thus it follows that 
$A_{(1+b)}=R_{(1+b)}[v,w]={R_{(1+b)}}^{[2]}$ and $A_{(1-b)}=R_{(1-b)}[u,w]={R_{(1-b)}}^{[2]}$. 
Since $(1+b)$ and $(1-b)$ are comaximal elements 
of $R$, $A=R[u,v,w]$ and $A_{\m}={R_{\m}}^{[2]}$ for every maximal ideal $\m$ of $R$. 

Now $B=R[X,Y,Z]=R[X,Y,w]$ and $w\in A$; so $rank$ $D$ $<3$. 
Setting $T=\frac{u}{a}$, we see that $A=R[aT,(1+b)T,w]$. By Theorems \ref{bcw} and \ref{sst}, 
$A=Sym_{R}(F \oplus P)$, where $F$ is a free $R$-module of rank $1$ and $P$ is a rank $1$ projective 
$R$-module given by the ideal $(a,1+b)R$, which is not principal. Hence $P$ is not stably free and 
so $A\neq R^{[2]}$ \\ 
\cite[Lemma 1.3]{EH}.
}
\end{ex}

\begin{rem}
{\em In Proposition \ref{DD}, we will see that over any Dedekind domain $R$, the kernel of a nice derivation 
of $R^{[3]}$ is generated by (at most) three elements.
}
\end{rem}

The following example shows that Part (ii) of Theorem \ref{Wang1} does not hold when $K$ is replaced by a PID $R$.

\begin{ex}\label{PIDex}
{\em Let $k$ be a field of characteristic zero, $R=k{[t]}=k^{[1]}$ and $B:=R[X,Y,Z](=R^{[3]})$. Let $D\in LND_R(B)$ be such 
that $$DX=0, \quad DY=X-t \quad \text{and} \quad DZ=X+t.$$ Let $A=Ker$ $D$ and $G:=(X-t)Z-(X+t)Y$. We will show that 
\begin{enumerate}
 \item [\rm (i)] $A=R[X,G]$.
 \item [\rm (ii)] $B\neq A^{[1]}$; in fact, $B$ is not $A$-flat.
 \item [\rm (iii)] $rank$ $D=2$.
\end{enumerate}

\begin{proof}
(i) Let $C:=R[X,G]$. We show that $C=A$. Clearly $C \subseteq A$. Set $f:=X-t$. 
Then $B_{f}={R[X,G,Y]}_{f}={C_{f}}^{[1]}$.
Hence, as both $C_{f}(\subseteq A_{f})$ and $A_{f}$ are factorially closed subrings of $B_{f}$ and as 
$\td_{C_{f}}{B_{f}}=1=\td_{A}{B}$, we have $C_{f}=A_{f}$.

Now $B/fB$ may be identified with $R[Y,Z](=R^{[2]})$. Clearly $C/fC=R^{[1]}$ and the image of $C/fC$ in $B/fB$ is 
$R[tY](=R^{[1]})$. Thus the natural map $C/fC \rightarrow B/fB$ is injective, i.e, $fB \cap C=fC$. 
Since $A$ is factorially closed in $B$, we also have $fB \cap A=fA$ and hence 
$fA \cap C=fB \cap A \cap C=fB \cap C=fC$. Therefore as $C_{f}=A_{f}$, we have $C=A$ by Lemma \ref{domain}. 

(ii) $(X-t,X+t)B$ is a prime ideal of height $2$ in $B$ and $(X-t,X+t)B \cap A=(X,t,G)A$ is a prime ideal of height 
$3$ in $A$, violating the going-down principle. Hence $B$ is not $A$-flat and therefore $B \neq A^{[1]}$.

(iii) Since $DX=0$, $rank$ $D < 3$. If $rank$ $D=1$, then clearly $B=A^{[1]}$ contradicting (ii). Hence $rank$ $D=2$. 
\end{proof}
}
\end{ex}

We now prove an extension of Theorem \ref{Wang1} over a PID.
\begin{thm}\label{PID}
Let $R$ be a PID containing $\bQ$ with field of fractions $L$ and \\ $B:=R[X,Y,Z]=R^{[3]}$.
Let $D(\neq0) \in LND_{R}(B)$, and $A:=Ker$ $D$. Suppose that $D$ is irreducible and $D^{2}X=D^{2}Y=D^{2}Z=0$. 
Then there exists a coordinate system $(U,V,W)$ of $B$ related to $(X,Y,Z)$ by a linear change such that the following hold:
\begin{enumerate}
 \item [\rm (i)] $A$ contains a nonzero linear form of $\{X,Y,Z\}$.
 \item [\rm (ii)] $rank$ $D\leq 2$. In particular, $A=R^{[2]}$.
 \item [\rm (iii)] $A=R[U,gV-fW]$, where $DV=f$, $DW=g$, and $f,g\in R[U]$ such that $gcd_{R[U]}(f,g)=1$.
 \item [\rm (iv)] Either $f$ and $g$ are comaximal in $B$ or they form a regular sequence in $B$. Moreover if they are 
 comaximal, (i.e., $D$ is fixed-point free) then $B=A^{[1]}$ and $rank$ $D=1$; and if they form a regular sequence, 
 then $B$ is not $A$-flat and $rank$ $D=2$. 
 \end{enumerate}
\end{thm}

\begin{proof}
(i) $D$ extends to an $LND$ of $L[X,Y,Z]$ which we denote by $\overline{D}$. By Theorem \ref{Wang1} there exists a coordinate 
system $(U,V^{'},W^{'})$ of $L[X,Y,Z]$ related to $(X,Y,Z)$ by a linear change and mutually coprime polynomials 
$p(U)$, $q(U)$ in $L[U]$ for which 
\begin{equation*}
\overline{D}=p(U)\frac{\partial}{\partial V^{'}}+q(U)\frac{\partial}{\partial W^{'}}.
\end{equation*}
Multiplying by a suitable nonzero element of $R$, we can assume $U \in R[X,Y,Z]$. 
Clearly $A=Ker$ $\overline{D} \cap R[X,Y,Z]$ and $U \in A$.
Moreover without loss of generality we can assume that there exist $l,m,n \in R$ with $gcd_R(l,m,n)=1$ such that
$U=lX+mY+nZ$. As $R$ is a PID, $(l,m,n)$ is a unimodular row of $R^{3}$ and hence can be completed to an 
invertible matrix $M \in GL_3(R)$. Let 
$
\begin{pmatrix}
U\\
V \\
W
\end{pmatrix}
=
M
\begin{pmatrix}
X \\Y\\ Z
\end{pmatrix}$.\\
Then $R[U,V,W]=R[X,Y,Z]$ and as $U \in A$, $A$ contains a nonzero linear form in ${X,Y,Z}$.

(ii) Follows from (i) and Lemma \ref{UFD}.

(iii) $R[U]$ is a UFD and $B=R[U,V,W]={R[U]}^{[2]}$. So $D$ is a locally nilpotent $R[U]$-derivation of $B$. 
Now the proof follows from Part (ii) of Lemma \ref{UFDWang}.

(iv) Since $B=R[U,V,W]={R[U]}^{[2]}$, the first part follows from Theorem \ref{bd97}. Moreover when $f$ and $g$ are 
comaximal in $B$, it also follows from Theorem \ref{bd97} that $B=A^{[1]}$. Hence in this case $rank$ $D=1$.

If $f$ and $g$ form a regular sequence in $B$ (and hence in $A$ since $A$ is factorially closed in $B$), 
$(f,g)B \cap A=(f,g,gV-fW)A$. But $(f,g,gV-fW)A$ is an ideal of height $3$, while $(f,g)B$ is an ideal of height $2$, 
violating the going-down principle. It follows that in this case $B$ is not $A$-flat. In this case indeed $rank$ $D=2$, 
or else if $rank$ $D=1$, we would have $B=A^{[1]}$.
\end{proof}

The proof of Theorem \ref{PID} shows the following: 
\begin{cor}\label{PIDflat}
With the notation as above, the following are equivalent:
\begin{enumerate}
 \item [\rm (i)] $B=A^{[1]}$.
 \item [\rm (ii)] $rank$ $D=1$.
 \item [\rm (iii)] $B$ is $A$-flat.
\end{enumerate}
\end{cor}
\begin{proof}
 (i)$\Leftrightarrow$(ii) and (ii)$\Rightarrow$(iii) are trivial. (iii)$\Rightarrow$(i) follows from Theorem \ref{PID}(iv). 
\end{proof}

As mentioned in the Introduction, Theorem \ref{PID} shows that the kernel of an irreducible nice derivation of $k^{[4]}$ 
of $rank\leq 3$ is $k^{[3]}$. More precisely, we have:
\begin{cor}\label{4variables}
Let $K$ be a field of characteristic zero and let $K[X_1,X_2,X_3,X_4]=K^{[4]}$. Let 
$D\in LND_{K}(K[X_1,X_2,X_3,X_4])$, be such that
$D$ is irreducible and $DX_1=0$ and $D^{2}X_i=0$ for $i=2,3,4$. Then $Ker$ $D = K^{[3]}$.
\end{cor}

By a result of Bhatwadekar and Daigle \cite[Proposition 4.13]{BD}, we know that over a Dedekind domain $R$ 
containing $\bQ$, the kernel of any locally nilpotent $R$-derivation of $R^{[3]}$ is necessarily finitely generated. 
We now show that if $D$ is a nice derivation, then the kernel is generated by at most three elements.

\begin{prop}\label{DD}
Let $R$ be a Dedekind domain containing $\bQ$ and $B:=R[X,Y,Z]=R^{[3]}$. Let $D\in LND_R(B)$ such that 
$D$ is irreducible and $D^{2}X=D^{2}Y=D^{2}Z=0$. Let $A:=Ker$ $D$. Then the following hold:
\begin{enumerate}
 \item [\rm (i)] $A$ is generated by at most 3 elements. 
 \item [\rm (ii)] Moreover, if $D$ is fixed-point free, then $rank$ $D$ $<3$ and $D$ has a slice. In particular, 
 $rank$ $D=1$.
\end{enumerate}
\end{prop}

\begin{proof}
(i) By Theorem \ref{PID}, $A_{\p}={R_{\p}}^{[2]}$ for all $\p \in Spec(R)$. 
Hence by Theorem \ref{bcw}, $A \cong Sym_R(Q)$ for some rank $2$ projective $R$-module $Q$. 
Since $R$ is a Dedekind domain by Theorem \ref{sst}, $Q \cong Q_{1} \oplus M$ where $Q_{1}$ is a rank $1$ 
projective $R$-module and $M$ is a free $R$-module of rank $1$. 
Again since $R$ is a Dedekind domain $Q_{1}$ is generated by at most 2 elements.
Hence $A$ is generated by at most 3 elements.

(ii) Now assume $D$ is fixed-point free. Let $DX=f_1$, $DY=f_2$ and $DZ=f_3$. Then, by Theorem \ref{st}, 
$B_{f_{i}}={A_{f_{i}}}^{[1]}$ for each $i \in \{1,2,3\}$. Since $(f_1,f_2,f_3)B=B$ we have 
$B_{\tilde{\p}}={A_{\tilde{\p}}}^{[1]}$, for each $\tilde{\p} \in Spec(A)$. Hence, by Theorem \ref{bcw}, 
$B=Sym_{A}(P)$, where $P$ is a projective $A$-module of rank $1$. Now for each $\p \in Spec(R)$, $P_{\p}$ is an 
$A_{\p}$-module and as $A_{\p}={R_{\p}}^{[2]}$, we have $P_{\p}$ is a free $A_{\p}$-module since $R_{\p}$ is a discrete 
valuation ring and hence extended from $R_{\p}$. Therefore, by Theorem \ref{gr}, $P$ is extended from $R$. 
Let $P=P_{1} \otimes_{R} A$, where $P_1$ is a projective $R$-module of rank $1$. Hence 
$B=Sym_{A}(P)=Sym_{R}(M \oplus Q_{1} \oplus P_{1})$, where $M$ is a free $R$-module of rank $1$. 
Since $B=R^{[3]}$, $M \oplus Q_{1} \oplus P_{1}$ is a free $R$-module of rank $3$ \cite[Lemma 1.3]{EH}. 
By Theorem \ref{bass}, $Q_{1} \oplus P_{1}$ is free of rank $2$. Let $M=Rf$ and set $S:=R[f]$. Then $B={R[f]}^{[2]}$ and 
as $f \in A$, we have $rank$ $D$ $<3$. Now $B=S^{[2]}$ and $D \in LND_{S}(B)$ such that $D$ is fixed-point free. 
Hence, by Theorem \ref{fpf}, $D$ has a slice.

Let $B=R[f,g,h](=R^{[3]})$ and $s\in B$ be such that $Ds=1$. Then by Theorem \ref{st}, $B(=S^{[2]})=A[s](=A^{[1]})$. 
Hence by Theorem \ref{ham}, $A=S^{[1]}$. Let $A=R[f,t]$. Then $B=R[f,g,h]=R[f,t,s]$ and $f,t \in A$. So $rank$ $D=1$.
\end{proof}

The following example shows that Theorem \ref{PID} does not extend to a higher-dimensional regular UFD, not even to $k^{[2]}$.

\begin{ex}\label{UFDex}
{\em
Let $k$ be a field of characteristic zero 
and $R=k[a,b]=k^{[2]}$. Let $B=R[X,Y,Z](=R^{[3]})$ and $D$($\neq0$)$\in LND_{R}(B)$ be such that
$$DX=b, \quad DY=-a \quad \text{and} \quad DZ=aX+bY.$$ Let $u=aX+bY$, $v=bZ-uX$, and $w=aZ+uY$. Then $Du=Dv=Dw=0$, 
$D$ is irreducible and $D^{2}X=D^{2}Y=D^{2}Z=0$. Let $A=Ker$ $D$. We show that 
\begin{enumerate}
 \item [\rm (i)] $A=R[u,v,w]$.
 \item [\rm (ii)] $A=R[U,V,W]/(bW-aV-U^{2})$, where $R[U,V,W]=R^{[3]}$ and hence $A \neq R^{[2]}$.
 \item [\rm (iii)] $rank$ $D=3$.
\end{enumerate}

\begin{proof}
(i) Let $C:=R[u,v,w]$. We show that $C=A$. Clearly $C \subseteq A$. Note that, $B_{a}={C_{a}}^{[1]}$, so $C_{a}$ is 
algebraically closed in $B_{a}$. But $A$ is algebraically closed in $B$. So $A_{a}=C_{a}$. Similarly $A_{b}=C_{b}$. 
Since $a,b$ is a regular sequence in $C$, $C_{a} \cap C_{b}=C$. Therefore $A \subseteq A_{a} \cap A_{b}=C_{a} \cap C_{b}=C$.

(ii) Let $\phi : R[U,V,W](= R^{[3]}) \twoheadrightarrow A$ 
be the $R$-algebra epimorphism such that $\phi(U)=u$, $\phi(V)=v$ and $\phi(W)=w$. Then $(bW-aV-U^{2}) \subseteq Ker$ $\phi$ 
and $bW-aV-U^{2}$ is an irreducible polynomial of the UFD $R[U,V,W]$. Now 
$\td_{R}{(R[U,V,W]/(bW-aV-U^{2}))}=2=\td_{R}{A}$. Hence $A \cong R[U,V,W]/(bW-aV-U^{2})$. 
Let $F=bW-aV-U^{2}$. Now $(\frac{\partial F}{\partial U},\frac{\partial F}{\partial V}, 
\frac{\partial F}{\partial W},F)R[U,V,W] \neq R[U,V,W]$. So $A$ is not a regular ring, in particular, $A \neq R^{[2]}$.

(iii) $rank$ $D$ $=3$ by Lemma \ref{UFD}.
\end{proof}
}
\end{ex}




\section{Quasi-nice Derivations}
In this section we discuss quasi-nice derivations. Let $k$ be a field of characteristic zero, $R$ a $k$-domain, $B:=R^{[n]}$ 
and $m$ be a positive integer $\leq n$. We shall call a quasi-nice $R$-derivation of $B$ to be {\it $m$-quasi} if, for some 
coordinate system $(T_1,T_2, \dots,T_n)$ of $B$, $D^{2}(T_i)=0$ for all $i \in \{1,\dots,m\}$. Thus for any two positive 
integers $r$ and $m$ such that $1\leq m< r\leq n$, it is easy to see that an $r$-quasi derivation is also an $m$-quasi 
derivation. We shall call an  $m$-quasi derivation to be {\it strictly $m$-quasi} if it is not $r$-quasi for any positive 
intger $r > m$.

Over a field $K$, Z. Wang \cite[Theorem 4.7 and Remark 5]{W} has proved the following result for $2$-quasi derivations.
\begin{thm}\label{Wang2}
Let $K$ be a field of characteristic zero 
and $K[X,Y,Z]=K^{[3]}$. Let $D(\neq0) \in LND_{K}(K[X,Y,Z])$  be such that $D$ is irreducible and 
 $D^{2}X=D^{2}Y=0$. Then one of the following holds:
\begin{enumerate}
\item[\rm (I)]There exists a coordinate system ($L_1,L_2,Z$) of $K[X,Y,Z]$, where $L_1$ and $L_2$ are linear forms in $X$ and 
$Y$ such that
\begin{enumerate}
\item[\rm (i)]$DL_1=0$.
\item[\rm (ii)]$DL_2 \in K[L_1]$.
\item[\rm (iii)]$DZ \in K[L_1,L_2]=K[X,Y]$.
\end{enumerate}
In this case, $rank$ $D$ can be either $1$ or $2$.
\item[\rm (II)]There exists a coordinate system ($V,X,Y$) of $K[X,Y,Z]$, such that $DV=0$ and $DX, DY \in K[V]$. 
In particular, $rank$ $D=1$.
\end{enumerate}
Conversely if $D \in Der_K(K[X,Y,Z])$ satisfies (I) or (II), then $D \in LND_K(K[X,Y,Z])$ and $D^{2}X=D^{2}Y=0$.
\end{thm}

The following two examples illustrate the cases $rank$ $D=1$ and $rank$ $D=2$ for Part (I) of Theorem \ref{Wang2}. 
\begin{ex}\label{rank1}
 {\em Let $D \in LND_K(K[X,Y,Z])$ be such that $DX=DY=0$ and $DZ=1$. Then $rank$ $D=1$.
 }
\end{ex}

\begin{ex}\label{rank2}
{\em Let $D \in LND_K(K[X,Y,Z])$ such that $$DX=0, DY=X, DZ=Y.$$ Setting $R=K[X]$, we see that $D \in LND_R(R[Y,Z])$ and $D$
is irreducible. By Theorem \ref{df}, 
$D=\frac{\partial F}{\partial Z}\frac{\partial}{\partial Y}-\frac{\partial F}{\partial Y}
\frac{\partial}{\partial Z}$ for some $F \in R[Y,Z]=K[X,Y,Z]$ such that $K(X)[Y,Z]=K(X)[F]^{[1]}$, $gcd_{R[Y,Z]}
(\frac{\partial F}{\partial Y},\frac{\partial F}{\partial Z})=1$. Moreover $Ker$ $D=R[F]=R^{[1]}$. Setting $F=XZ-\frac
{Y^{2}}{2}$ we see $\frac{\partial F}{\partial Y}=-Y=-DZ$ and $\frac{\partial F}{\partial Z}=X=DY$.

Therefore $Ker$ $D=K[X,F]$. But $F$ is not a coordinate in $K[X,Y,Z]$ since $(\frac{\partial F}{\partial X},\frac{\partial F}
{\partial Y},\frac{\partial F}{\partial Z})K[X,Y,Z]\neq K[X,Y,Z]$. So $rank$ $D=2$.
 }
\end{ex}

We now address Question $2$ of the Introduction, which gives a partial generalisation of Theorem \ref{Wang2}.
\begin{prop}\label{partI}
Let $R$ be a PID containing $\mathbb{Q}$ with field of fractions $K$.
Let $D \in LND_{R}(R[X,Y,Z])$, where $R[X,Y,Z]=R^{[3]}$ such that $D$ is irreducible and $D^{2}X=D^{2}Y=0$. 
Let $\overline{D} \in LND(K[X,Y,Z])$ denote the extension of $D$ to $K[X,Y,Z]$. Let $A :=Ker$ $D$.
Suppose $\overline{D}$ satisfies condition (I) of Theorem \ref{Wang2}.
Then the following hold:
\begin{enumerate}
 \item [\rm (i)] $rank$ $D<3$.
 \item [\rm (ii)] There exists a coordinate system $(L_1,L_2,Z)$ of $B$, such that $L_1,L_2$ are linear forms
 in $X$ and $Y$, $DL_1=0$, $DL_2 \in R[L_1]$ and $DZ \in R[L_1,L_2]=R[X,Y]$. Moreover, $A=R[L_1,bZ+f(L_2)]$, where 
 $b \in R[L_1]$ and $f(L_2) \in R[L_1,L_2]$.
\end{enumerate}
\end{prop}

\begin{proof}
(i) Let $(\overline{L_1},\overline{L_2},Z)$ be the coordinate system of $K[X,Y,Z]$ such that $\overline{D}$ satisfies condition (I) 
of Theorem \ref{Wang2}. Multiplying by a suitable nonzero constant from $R$, we can assume $\overline{L_1} \in R[X,Y]$. 
Let $\overline{L_1}=aX+bY$ where $a,b \in R$. 
Without loss of generality we can assume $gcd_{R}(a,b)=1$. Since $R$ is a PID, 
$(a,b,0)$ is a unimodular row in $R^{3}$
and hence can be completed to an invertible matrix (say $N$) in $GL_3(R)$.
Thus $\overline{L_1}$ is a coordinate in $R[X,Y,Z]$. As $\overline{L_1} \in Ker D = Ker \overline{D} \cap R[X,Y,Z]$, $rank$ $D$ is 
at most $2$ and 
hence $rank$ $D\leq 2< 3$.

(ii) Now set $L_1=\overline{L_1}$. Since $gcd_{R}(a,b)=1$, there exist $c,d \in R$ such that $ad-bc=1$. Hence we can choose 
$N$ as 
$
\begin{pmatrix}
a & b & 0\\
c & d & 0\\
0 & 0 & 1
\end{pmatrix}$. Then $N 
\begin{pmatrix}
X \\
Y \\
Z
\end{pmatrix}=
\begin{pmatrix}
L_1 \\
L_2 \\
Z
\end{pmatrix}$. \\
Now the proof follows from Part (i) of Lemma \ref{UFDWang}. 
\end{proof}

With the notation as above, if $\overline{D}$ satisfies condition (II) of Theorem \ref{Wang2}, $rank$ $D$ need not be $1$. The 
following example shows that $rank$ $D$ can even be $3$.

%

\begin{ex}\label{2var}
{\em Let $k$ be a field of characteristic zero, $R=k[t](=k^{[1]})$ with field of fractions $L$ and $B:=R[X,Y,Z](=R^{[3]})$. 
Let $D \in LND_{R}(B)$ be defined by $$DX=t, \quad DY=tZ+X^{2} \quad \text{and} \quad DZ=-2X.$$ Then $D$ is irreducible 
and $D^{2}X=D^{2}Y=0$. Let $\overline{D}$ denote the extension of $D$ to $L[X,Y,Z]$. Let $F=-GX+tY$ where $G=tZ+X^{2}$.  
Then $F^{2}-G^{3}=tH$, where $H=tY^{2}-2tX^{2}Z^{2}-2tXYZ-2X^{3}Y-X^{4}Z-t^{2}Z^{3} \in R[X,Y,Z]$. Set $C:=R[F,G,H]$. 
We show that 
\begin{enumerate}
 \item [\rm (i)] $\overline{D}$ satisfies condition (II) of Theorem \ref{Wang2}.
 \item [\rm (ii)] Then $C \cong R[U,V,W]/(U^{2}-V^{3}-tW)$, where $R[U,V,W]=R^{[3]}$ and hence 
 $C \neq R^{[2]}$.
 \item [\rm (iii)] $Ker$ $D=C$.
 \item [\rm (iv)] $rank$ $D=3$.
\end{enumerate}

\begin{proof}
(i) $L[X,Y,Z]=L[X,Y,G]$, $DG=\overline{D}G=0$ and $\overline{D}X,\overline{D}Y \in L[G]$. By Theorem \ref{df}, 
$Ker$ $\overline{D}=L[F,G](=L^{[2]})$.

(ii) Consider the $R$-algebra epimorphism $\phi:R[U,V,W] \twoheadrightarrow R[F,G,H]$($=C$), given by $\phi(U)=F$, 
$\phi(V)=G$ and $\phi(W)=H$. Clearly ($U^{2}-V^{3}-tW$) $\subseteq Ker$ $\phi$. Since $U^{2}-V^{3}-tW$ is an irreducible 
polynomial in $R[U,V,W]=k[t,U,V,W]=k^{[4]}$,  and $\td_{k} R[U,V,W]/(U^{2}-V^{3}-tW)=3$ we have 
$Ker$ $\phi= (U^{2}-V^{3}-tW)$ and hence $C \cong R[U,V,W]/(U^{2}-V^{3}-tW)$.

Set $f:=U^{2}-V^{3}-tW$. Then $C$ is not regular, since $(\frac{\partial f}{\partial t},\frac{\partial f}{\partial U},
 \frac{\partial f}{\partial V},\frac{\partial f}{\partial W},f) \subseteq (t,U,V,W)k[t,U,V,W]$. Thus $C$ is not regular; in 
 particular, $C \neq R^{[2]}$.

(iii) Let $A:=Ker$ $D$. Then $A=L[F,G] \cap R[X,Y,Z]$. We note that since $Ker~\overline{D}$ is factorially closed, 
$H \in Ker$ $\overline{D}$ and hence $H \in A$.  

$C_{t}=R_{t}[F,G,H]=R_{t}[F,G]$. Also $R_{t}[X,Y,Z]=R_{t}[X,Y,G]=R_{t}[X,F,G]$. $D$ extends to a locally nilpotent 
$R_{t}$-derivation (say $\tilde{D}$) of $R_{t}[X,Y,Z]$ and $\tilde{D}X \in {R_{t}}^{*}$. So by Theorem \ref{st}, 
$C_{t}=A_{t}$.

Clearly $C \subseteq A$. By Lemma \ref{domain}, it is enough to show that the map 
$C/tC \rightarrow A/tA$ is injective. Since $A$ is factorially closed in $B$, there exists an inclusion 
$A/tA \hookrightarrow B/tB$. So we will be done if we show the composite map 
 $\psi:C/tC \rightarrow B/tB$ is injective. For $g \in B$, let $\overline{g}$ denote the image of $g$ in $B/tB$.
 In $\psi(C/tC)$, $\overline{G}=\overline{X}^{2}$, $\overline{F}=-\overline{X}^{3}$ and 
$\overline{H}=-2\overline{X}^{3}\overline{Y}-\overline{X}^{4}\overline{Z}$. Since $\overline{X}$ and $\overline{Z}$ are algebraically 
independent over $k$, $\td_{k}\psi(C/tC)=2$. From (ii) it follows that $C/tC$ is an integral domain and 
$\td_{k}C/tC=\td_{k}\psi(C/tC)=2$. Hence $\psi$ is injective. So $C/tC \hookrightarrow A/tA$ and hence $C=A$ as desired.

(iv) By Lemma \ref{UFD}, $rank$ $D=3$.
\end{proof}
}
\end{ex}

Over a Dedekind domain $R$, we have the following generalisation of Proposition \ref{DD} and Proposition \ref{partI}.
\begin{prop}\label{dd2}
Let $R$ be a Dedekind domain containing $\bQ$ with field of fractions $K$, and $B:=R[X,Y,Z](=R^{[3]})$. Let $D \in LND_R(B)$ 
be irreducible and $D^{2}X=D^{2}Y=0$ and $\overline{D}$ denote the extension of $D$ to $K[X,Y,Z]$. Let $A:=Ker$ $D$. 
If $\overline{D}$ satisfies condition (I) of Theorem \ref{Wang2}, then the following hold:
\begin{enumerate}
 \item [\rm (i)] $A$ is generated by at most $3$ elements.
 \item [\rm (ii)] Moreover, if $D$ is fixed-point free, then $rank$ $D$ $<3$ and $D$ has a slice. In particular,
 $rank$ $D=1$.
\end{enumerate}
\end{prop}

\begin{proof}
(i) By Proposition \ref{partI} and Lemma \ref{UFD}, $A_{\p}={R_{\p}}^{[2]}$ for all $\p \in Spec(R)$. Now the proof follows 
from the proof of Part (i) of Proposition \ref{DD}. 

(ii) For each $\p \in Spec(R)$, let $D_{\p}$ denote the extension of $D$ to $B_{\p}$. Then by Proposition \ref{partI} and 
Theorem \ref{fpf}, $D_{\p}$ has a slice. Thus $B_{\p}={A_{\p}}^{[1]}$ for all $\p \in Spec(R)$. Now the proof follows from the 
proof of Part (ii) of Proposition \ref{DD}.
\end{proof}

\begin{rem}\label{ufdex2}
{\em Example \ref{UFDex} shows that Proposition \ref{partI} does not extend to a higher-dimensional UFD, not even to 
$k^{[2]}$, where $k$ is a field of characteristic zero. In fact, in that example, taking $L_1=u$ and $L_2=cX+dY$ 
for some $c,d \in k[a,b]$ such that $ad-bc\neq 0$, we find that $\overline{D}$ satisfies condition (I) of Theorem \ref{Wang2} and 
considering the coordinate system $(u,X,Z)$, of $K[X,Y,Z]$ where $K$ is the field of fractions of $k[a,b]$, we also see 
that $\overline{D}$ satisfies condition (II).
}
\end{rem}

The following theorem of Daigle shows that over a field $k$ of characteristic zero, there does not exist any strictly 
$1$-quasi derivation of $k^{[3]}$ \cite[Theorem 5.1]{Da}.
\begin{thm}\label{d10}
Let $k$ be a field of characteristic zero , $B=k^{[3]}$ and $D:B\rightarrow B$ be an irreducible locally nilpotent derivation. 
We assume that some variable $Y$ of $B$ satisfies $DY \neq 0$ and $D^{2}Y=0$. Then there exist $X,Z$ such that 
$$ B=k[X,Y,Z], \quad DX=0, \quad DY \in k[X] \quad \text{and} \quad DZ \in k[X,Y].$$
\end{thm}


We now present an example of a strictly $1$-quasi derivation of $R^{[3]}$ over a PID $R$ containing $\bQ$. Thus Theorem 
\ref{d10} does not extend to a PID. 
\begin{ex}\label{1var}
{\em Let $k$ be a field of characteristic zero, $R=k[t](=k^{[1]})$ and $B=R[X,Y,Z](=R^{[3]})$. Let $D \in LND_{R}(B)$ be 
such that $$ DX=t, \quad DY=X \quad \text{and} \quad DZ=Y.$$ Then $D$ is irreducible and $D^{2}X=0$.  Let $F:=2tY-X^{2}$, 
$G=3t^{2}Z-3tXY+X^{3}$ and $H=8tY^{3}+9t^{2}Z^{2}-18tXYZ-3X^{2}Y^{2}+6X^{3}Z$. Then $F^{3}+G^{2}=t^{2}H$. Set $C:=R[F,G,H]$. 
We now show the following:
\begin{enumerate}
 \item [\rm (i)] Then $C \cong R[U,V,W]/(U^{3}+V^{2}-t^{2}W)$, \\ where $R[U,V,W]=R^{[3]}$.
 \item [\rm (ii)] $Ker$ $D=C$.
 \item [\rm (iii)] There does not exist any coordinate system $(U_1,U_2,U_3)$ of $B$, such that $D^{2}(U_1)=D^{2}(U_2)=0$.
\end{enumerate}

\begin{proof}
(i) Consider the $R$-algebra epimorphism $\phi:R[U,V,W] \twoheadrightarrow R[F,G,H](=C)$, given by $\phi(U)=F$, $\phi(V)=G$ 
and $\phi(W)=H$. Clearly $(U^{3}+V^{2}-t^{2}W) \subseteq Ker$ $\phi$. Since $U^{3}+V^{2}-t^{2}W$ is an irreducible polynomial 
in $R[U,V,W]=k[t,U,V,W]=k^{[4]}$, and $\td_{k} R[U,V,W]/(U^{3}+V^{2}-t^{2}W)=3$ we have 
$Ker$ $\phi= (U^{3}+V^{2}-t^{2}W)$ and hence $C \cong R[U,V,W]/(U^{3}+V^{2}-t^{2}W)$. 

(ii) Let $A:=Ker$ $D$. Since $A$ is factorially closed in $B$, $H \in A$. $C_{t}=R_{t}[F,G,H]=R_{t}[F,G]$. 
Also $R_{t}[X,Y,Z]=R_{t}[X,F,G]$. $D$ extends to a locally nilpotent $R_{t}$-derivation (say $\tilde{D}$) of 
$R_{t}[X,Y,Z]$ and $\tilde{D}X \in {R_{t}}^{*}$. So by Theorem \ref{st}, 
$C_{t}=A_{t}$.

Clearly $C \subseteq A$. By Lemma \ref{domain}, it is enough to show that the map 
$C/tC \rightarrow A/tA$ is injective. Since $A$ is factorially closed in $B$, there exists an inclusion 
$A/tA \hookrightarrow B/tB$. So we will be done if we show the composite map 
 $\psi:C/tC \rightarrow B/tB$ is injective. For $g \in B$, let $\overline{g}$ denote the image of $g$ in $B/tB$.
 In $\psi(C/tC)$, $\overline{F}=-\overline{X}^{2}$, $\overline{G}=\overline{X}^{3}$ and 
$\overline{H}=-3\overline{X^{2}}\overline{Y^{2}}+6\overline{X^{3}}\overline{Z}$. Since $\overline{X}$, $\overline{Y}$ and $\overline{Z}$ are algebraically 
independent over $k$, $\td_{k}\psi(C/tC)=2$. From (ii) it follows that $C/tC$ is an integral domain and 
$\td_{k}C/tC=\td_{k}\psi(C/tC)=2$. Hence $\psi$ is injective. So $C/tC \hookrightarrow A/tA$ and hence $C=A$ as desired.

(iii) For $f \in B$, let $\alpha:=$ coefficient of $X$ in $f$, $\beta:=$ coefficient of $Y$ in $f$ and $deg$ $f:=$ 
total degree of $f$. Then $\frac{\partial f}{\partial X}(0,0,0)=\alpha$ and $\frac{\partial f}{\partial Y}(0,0,0)=\beta$. 
If $f \in A$, let $f=p(F,G,H)$ for some $p \in R^{[3]}$. Then $\frac{\partial f}{\partial X}=\frac{\partial p}{\partial F}\frac{\partial F}{\partial X}+
\frac{\partial p}{\partial G}\frac{\partial G}{\partial X}+\frac{\partial p}{\partial H}\frac{\partial H}{\partial X}$. 
We also have $$\frac{\partial F}{\partial X}=-2X, \quad \frac{\partial G}{\partial X}=-3tY+3X^{2} \quad \text{and} \quad
\frac{\partial H}{\partial X}=-18tYZ-6XY^{2}+18X^{2}Z.$$ Thus $\alpha=0$. Again, since 
$$\frac{\partial F}{\partial Y}=-2t, \quad \frac{\partial G}{\partial Y}=-3tX, \quad \text{and} \quad 
\frac{\partial H}{\partial Y}=16tY-18tXZ-6X^{2}Y,$$ similarly we have $\beta=\lambda t$, for some $\lambda \in R$. 
Let $U$ be a coordinate in $B$ such that $D^{2}U=0$. Since $U$ is a 
coordinate, there exist $a,b,c \in R$, not all $0$ and $V \in B$ with $deg$ $V\geqslant 2$ and no linear term, such that 
$U=aX+bY+cZ+V$. Then 
$DU=at+bX+cY+(\frac{\partial V}{\partial X})t+(\frac{\partial V}{\partial Y})X+(\frac{\partial V}{\partial Z})Y$.  
Let $\gamma:=$ coefficient of $X$ in $\frac{\partial V}{\partial X}$ and $\delta:=$ coefficient of $Y$ in 
$\frac{\partial V}{\partial X}$. Then coefficient of $X$ in $DU=b+\gamma t$ and coefficient of $Y$ in $DU=c+\delta t$ 
(We can ignore terms from $(\frac{\partial V}{\partial Y})X$ and $(\frac{\partial V}{\partial Z})Y$ since neither of them 
has any linear term). Thus $b+\gamma t=0$ and $c+ \beta t=\lambda t$ for some $\lambda \in R$. Therefore,
$b \in (t)R$ and $c \in (t)R$.

Let $(U_1,U_2,U_3)$ be a coordinate system of $B$ such that $D^{2}(U_1)=D^{2}(U_2)=0$ and let 
$U_i=a_{i}X+b_{i}Y+c_{i}Z+V_{i}$ where $a_{i},b_{i},c_{i}\in R$, not all $0$ and $V_{i} \in B$ with $deg$ $V_{i} \geqslant 2$ 
and no linear term, for $i=1,2$. Thus for each $i$, we have $b_{i},c_{i} \in (t)R$. For $f \in B$, let $\overline{f}$ denote its 
image in $B/tB(=k^{[3]})$. Since $B/tB=k[\overline{U_1},\overline{U_2},\overline{W}]$, $\overline{U_1},\overline{U_2}$ form a partial coordinate 
system in $B/tB$. But $\overline{U_{i}}=\overline{a_{i}}\overline{X}+V_{i}$ and since $\overline{V_{i}}$ has no linear term, 
$\overline{a_{i}}\neq 0$, for each $i$. Then $\overline{a_{2}}\overline{U_{1}}-\overline{a_{1}}\overline{U_{2}}$ has no linear term, but it is a 
coordinate in $B/tB$. Hence we have a contradiction.
\end{proof}

}
\end{ex}

\smallskip
\noindent

%
%
%
%
%

%
%

\noindent
{\small
{\bf Acknowledgement:} The authors thank Professor Amartya K. Dutta for asking Questions 1 and 2, going through the earlier 
drafts and suggesting improvements.The second author also acknowledges Department of Science and Technology for their Swarnajayanti award.}
}

{\small{

}}

\begin{thebibliography}{9999}
\bibitem{AEH} S.S. Abhyankar, P. Eakin and W. Heinzer, 
{\it On the uniqueness of the coefficient ring in a polynomial ring},
J. Algebra {\bf 23} (1972) 310--342.
\bibitem{BCW} H. Bass, E.H. Connell, D.L. Wright, 
{\it Locally polynomial algebras are symmetric algebras}, 
Invent. Math. {\bf 38} (1977) 279-299.
\bibitem{BD} S.M. Bhatwadekar, D. Daigle,
{\it On finite generation of kernels of locally nilpotent $R$-derivations of $R[X,Y,Z]$},
J. Algebra {\bf 322}(9) (2009) 2915-2926.
\bibitem{BD97} S.M. Bhatwadekar, A.K. Dutta,
{\it Kernel of locally nilpotent $R$-derivations of $R[X,Y]$},
Trans. Amer. Math. Soc. {\bf 349}(8) (1997) 3303-3319.
\bibitem{BGL} S.M. Bhatwadekar, N. Gupta, S.A. Lokhande,
{\it Some $K$-theoritic properties of the kernel of a locally nilpotent derivation on $k[X_1,\dots,X_4]$},
Trans. Amer. Math. Soc. {\bf 369} (2017), 341-363.
\bibitem{Da} D. Daigle,
{\it Triangular Derivations of $k[X,Y,Z]$},
Journal of Pure and Applied Algebra {\bf 214} (2010) 1173-1180.
\bibitem{DF} D. Daigle, G. Freudenburg,
{\it Locally nilpotent derivatons over a UFD and an application to rank two locally nilpotent derivations 
of $k[X_1,\dots,X_n]$}, 
J. Algebra {\bf 204} (1998) 353-371.
\bibitem{DF01} D. Daigle, G. Freudenburg,
{\it A note on triangular derivations of $k[X_1,X_2,X_3,X_4]$},
Proc. Amer. Math. Soc {\bf 129(3)} (2001) 657-662.
\bibitem{EH} P. Eakin, W. Heinzer,
{\it A cancellation problem for rings},
Conference on Commutative Algebra (Univ. Kansas, Lawrence, Kan., 1972), eds. J.W. Brewer and E.A. Rutter,  
Lecture notes in Math., Vol. 311, Springer, Berlin (1973) 61-77.
\bibitem{ST} G. Freudenburg, 
{\it Algebraic Theory of Locally Nilpotent Derivations}, 
Springer-Verlag Berlin Heidelberg (2006).
\bibitem{HAM} E. Hamann,
{\it On the $R$-Invariance of $R[X]$},
J. Algebra {\bf 35} (1975) 1-16.
\bibitem{SST} F. Ischebeck, R. Rao, 
{\it Ideals and Reality}, Springer-Verlag Berlin Heidelberg (2005).
\bibitem{SUS} A.A. Suslin,
{\it Locally polynomial rings and symmetric algebras (Russian)},
Izv. Akad. Nauk SSSR Ser. Mat. {\bf 41}(3) (1977) 503-515.
\bibitem{W} Z. Wang, 
{\it Homogenization of locally nilpotent derivations and an application to $k[X,Y,Z]$}, 
Journal of Pure and Applied Algebra {\bf 196} (2005) 323-337.
\end{thebibliography}
\end{document}